\documentclass{IEEEtran}
\usepackage{amssymb,amsthm}
\usepackage{bbm} 
\usepackage[utf8]{inputenc}
\usepackage{tikz}
\let\oequation\equation
\let\oeqend\endequation
\makeatletter
\def\@thiseqno{NOTHING}
\def\@myendequation{\eqno \hbox{\normalfont \normalcolor \@thiseqno}$$\@ignoretrue}
\def\@eeequation[#1]{\let\@ooeqend\@myendequation
$$\edef\@thiseqno{#1}\protected@edef\@currentlabel{\csname p@equation\endcsname#1}}%
\renewenvironment{equation}{%
\let\@ooeqend\oeqend
\@ifnextchar[\@eeequation\oequation}
{\@ooeqend}
\makeatother
\let\widebar\overline

\newtheorem{theorem}{Theorem}

\newtheorem{proposition}[theorem]{Proposition}

\newtheorem*{assumption}{Assumption}
\newtheorem{definition}[theorem]{Definition}
\newtheorem{oracle}[theorem]{Oracle}
\newtheorem{algorithm}{Algorithm}
\def\Q{\mathcal Q}

\def\R{\mathbb R}
\def\E{\mathbb E}
\def\e{\mathbbm e}
\def\O{\R_{\ge}}
\def\B{\mathcal B}
\def\S{\mathcal S}
\def\PSO{\mathord{\mbox{\sf ptO}}}
\def\HSO{\mathord{\mbox{\sf hpO}}}
\def\wPSO{\mathord{\mbox{\sf ptO}}^*}
\def\wHSO{\mathord{\mbox{\sf hpO}}^*}
\def\conv{\mathord{\mbox{\sf conv}}}
\def\eps{\varepsilon}
\newcommand\dispqed{\eqno \hbox{\normalfont\normalcolor\qedsymbol}}

\title{\bf Inner approximation algorithm for
solving linear multiobjective optimization problems}
\author{Laszlo Csirmaz%
\thanks{Central European University, Budapest}%
\thanks{e-mail:\tt~csirmaz@renyi.hu }}
\IEEEspecialpapernotice{
Dedicated to the memory of Fantisek Mat\'u\v s}

\begin{document}
\maketitle

\begin{abstract}

Benson's outer approximation algorithm and its variants are the most
frequently used methods for solving linear multiobjective optimization
problems. These algorithms have two intertwined parts:
single-objective linear optimization on one hand, and a combinatorial part
closely related to vertex enumeration on the other. Their separation
provides a deeper insight into Benson's algorithm, and points toward a dual
approach. Two skeletal algorithms are defined which focus on the
combinatorial part. Using different single-objective optimization problems
-- called oracle calls -- yield different algorithms, such as a sequential
convex hull algorithm, another version of Benson's algorithm with the 
theoretically best possible iteration count, the dual algorithm of 
Ehrgott, L\"ohne and Shao \cite{l-dual}, and 
the new algorithm. The new algorithm has several
advantages. First, the corresponding single-objective optimization problem
uses the original constraints without adding any extra variables or
constraints. Second, its iteration count meets the theoretically best
possible one. As a dual algorithm, it is sequential: in each iteration it
produces an extremal solution, thus can be aborted when a satisfactory
solution is found. The Pareto front can be ``probed'' or ``scanned'' from
several directions at any moment without adversely affecting the efficiency.
Finally, it is well suited to handle highly degenerate problems where there
are many linear dependencies among the constraints. On problems with ten or
more objectives the implementation shows a significant increase in
efficiency compared to {\em Bensolve} -- due to the reduced number of
iterations and the improved combinatorial handling.
\end{abstract}


\section{Introduction}\label{sec:intro}

Our notation is standard and mainly follows that of \cite{l-dual,loehne2}.
The transpose of the matrix $M$ is denoted by $M^T$. Vectors
are usually denoted by small letters and are considered as single column
matrices. For two vectors $x$ and $y$ of the same dimension, $xy$ denotes
their inner product, which is the same as the matrix product $x^Ty$.  The
$i$-th coordinate of $x$ is denoted by $x_i$, and $x\le y$ means that for
every coordinate $i$, $x_i\le y_i$. The non-negative orthant $\O^n$ is
the collection of vectors $x\in\R^n$ with $x\ge 0$, that is, vectors whose
coordinates are non-negative numbers.

For an introduction to higher dimensional polytopes see \cite{ziegler}, and
for a description of the double description method and its variants, consult
\cite{avis-fukuda}. Linear multiobjective optimization problems, methods,
and algorithms are discussed in \cite{l-dual} and the references therein.

\subsection{The MOLP problem}\label{subsec:MOLP}

Given positive integers $n$, $m$, and $p$, the $m\times n$ matrix $A$ maps
the {\em problem space $\R^n$} to $\R^m$, and the $p\times n$ matrix $P$
maps the problem space $\R^n$ to the {\em objective space $\R^p$}.  For
better clarity we use $x$ to denote points of the problem space $\R^n$,
while $y$ denotes points in the objective space $\R^p$.
In the problem space a convex closed polyhedral set $\mathcal A$ is
specified by a collection of linear constraints. For simplicity we assume
that the 
constraints are given in the following special format. This format will 
only be used in Section \ref{sec:oracle-calls}.
\begin{equation}\label{eq:A}
   \mathcal A = \{ x\in\R^n:\, Ax=c, \, x \ge 0 \,\},
\end{equation}
where $c\in\R^m$ is a fixed vector.
The $p$-dimensional {\em linear projection} of $\mathcal A$ is given by the
$p\times n$ matrix $P$ is
\begin{equation}\label{eq:Q}
    \Q = P\mathcal A = \{ Px \,:\, x\in\mathcal A\,\}.
\end{equation}
Using this notation, the {\em multiobjective linear optimization problem} can
be cast as follows:
\begin{equation}[MOLP]\label{eq:1}
\mbox{find }~~  \min\nolimits_y \, \{ y:\,y\in\Q\,\},
\end{equation}
where minimization is understood with respect to the coordinate-wise ordering
of $\R^p$. The point $\hat y\in\mathcal Q$ is {\em non-dominated} or {\em 
Pareto optimal}, if no $y
\le \hat y$ different from $\hat y$ is in $\mathcal Q$; and it is {\em weakly
non-dominated} if no $y < \hat y$ is in $\mathcal Q$. Solving the
multiobjective optimization problem is to find (a description of) all 
non-dominated vectors $\hat y$ together with the corresponding {\rm
pre-images} $\hat x\in\R^n$ such that $\hat y = P\hat x$. 

Let $\Q^+ = \Q + \O^p$, the Minkowski sum of $\Q$ and the non-negative
orthant of $\R^p$, see \cite{ziegler}. It follows easily from the
definitions, but see also \cite{loehne3,l-dual,loehne4},
that non-dominated points of $\Q$ and of $\Q^+$ are the same. The weakly
non-dominated points of $\Q^+$ form its {\em Pareto front}. Figure
\ref{fig:pareto} illustrates non-dominated (solid line), and
weakly non-dominated points (solid and dashed line) of $\Q^+$ when a) all
objectives are bounded from below, and when b) the first
objective is not bounded. $\Q^+$ is the unbounded light gray area
extending $\mathcal Q$.

\begin{figure}
\hfil\begin{tikzpicture}
\fill[black!4!white] (0.5,1)--(0.5,3)--(3.5,3)--(3.2,0.5)--(1,0.5)--cycle;
\draw[line width=1pt,color=black!30!white,fill=black!10!white] (0.5,1)--(0.5,2)--(2.8,2.5)
   --(3.2,1.8)--(2.5,0.8)--(1,0.5)--cycle;
\draw[ultra thin,->] (0,0)--(0,3.5);
\draw[ultra thin,->] (0,0)--(3.6,0);
\draw (1.8,1.5) node {$\mathcal Q$};
\draw[line width=1.5pt] (0.5,1)--(1,0.5);
\draw (0.5,1) node {$\bullet$} (1,0.5) node {$\bullet$};
\draw[line width=1.5pt,dashed,->] (0.5,1)--(0.5,3.2);
\draw[line width=1.5pt,dashed,->] (1,0.5)--(3.4,0.5);
\draw (1.7,-0.3) node {a)};
\end{tikzpicture}
\quad\hfil\quad
\begin{tikzpicture}
\fill[black!4!white] (0.5,1)--(0.5,3)--(3.5,3)--(3.2,0.5)--(1,0.5)--cycle;
\fill[fill=black!10!white] (-0.2,3.2)--(0.5,1)--(1,0.5)
    --(2.5,0.8)--(3.4,3.2)--cycle;
\draw[line width=1pt,color=black!30!white,->] (1,0.5)--(2.5,0.8)--(3.45,3.31);
\draw[ultra thin,->] (0,0)--(0,3.5);
\draw[ultra thin,->] (0,0)--(3.7,0);
\draw (1.8,1.8) node {$\mathcal Q$};
\draw[line width=1.5pt] (0.5,1)--(1,0.5);
\draw (0.5,1) node {$\bullet$} (1,0.5) node {$\bullet$};
\draw[line width=1.5pt,dashed,->] (1,0.5)--(3.55,0.5);
\draw[line width=1.5pt,->] (0.5,1)--(-0.25,3.3);
\draw (1.7,-0.3) node {b)};
\end{tikzpicture}\hfil
\caption{Pareto front with a) bounded, b) unbounded objectives}\label{fig:pareto}
\end{figure}
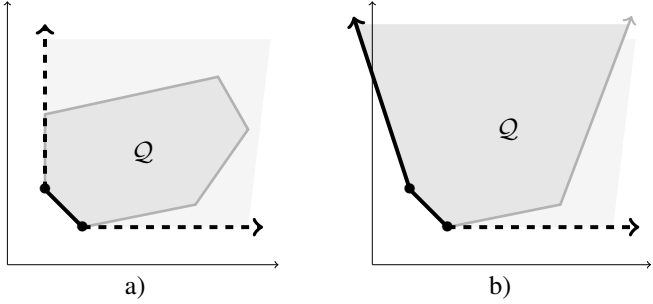
 
\subsection{Facial structure of polytopes}\label{subsec:facial}

Let us recall some facts concerning the facial structure of $n$-dimensional
convex closed polytopes. A {\em face} of such a polytope $\B
\subset  \R^n$ is the intersection of $\B$ and some closed
halfspace (including the empty set and the whole $\B$). 
Faces of dimension zero, one, $n-2$, and $n-1$ are called {\em
vertex, edge, ridge}, and {\em facet}, respectively.

An $n$-dimensional halfspace is specified as $H=\{x\in\R^n: xh\ge M\}$,
where $h\in\R^n$ is a non-null vector (normal), and $M$ is a scalar
(intercept). The {\em positive side} of $H$ is the open halfspace
$\{x\in\R^n: xh>M\}$; the negative side is defined similarly.
Each facet of $\B$ is identified with the halfspace which contains 
$\B$ and whose boundary intersects $\B$ in that facet. 
$\B$ is
just the intersection of all halfspaces corresponding to its facets.

The halfspace $H$ is {\em supporting} if it contains $\B$, and there is a
boundary point of $\B$ on the boundary of $H$.
The boundary hyperplane of a supporting halfspace intersects $\B$ in one
of its faces. All boundary points of $\B$ are in the relative interior of
exactly one face.

A {\em recession direction}, or {\em ray} of $\mathcal B$ is a vector 
$d\in \R^n$ such that $x+\lambda d \in \mathcal B$ for all real 
$\lambda\ge 0$. $d$ is {\em extreme} if whenever $d=d_1+d_2$
for two recession directions $d_1$ and $d_2$, then both $d_1$ and 
$d_2$ are non-negative multiples of $d$.

\subsection{Working with unbounded polytopes}\label{subsec:ideal}

When $\mathcal B$ is unbounded but does not contain a complete line -- which
will always be the case in this paper --, Burton and Ozlen's
``oriented projective geometry'' can be used
\cite{projective}. Intuitively this means that rays are represented by ideal
points, extreme rays are the ideal vertices which lie on a single
ideal facet determining the {\em ideal hyperplane}. Notions
of ordering and convexity can be extended to these objects seamlessly even
from computational point of view. In particular, all non-ideal points are on
the positive side of the ideal hyperplane. Thus in theoretical considerations,
without loss of generality, $\mathcal B$ can be assumed to be bounded.

\subsection{Assumptions on the MOLP problem}\label{subsec:assumptions}

In order that we could focus on the main points, we make some simplifying
assumptions on the \ref{eq:1} problem to be solved. The main restriction is that
all objectives are bounded from 
below. From this it follows immediately that neither $\Q$ nor $\Q^+$
contains a complete line;
and that the Pareto optimal solutions are the bounded faces of dimension
$p-1$ and less of $\Q^+$ as indicated on Figure \ref{fig:pareto}$\,$a).
One can relax this restriction at the expense of computing the extreme rays
of $\Q$ first (checking along that $\Q^+$ does not contain a complete
line), as is done by the software package Bensolve \cite{bensolve}. Further
discussions are postponed to Section \ref{sec:remarks}, where we also extend
our results to the case
where the ordering is given by some cone different from $\O^p$.

\begin{assumption}\rm
The optimization problem (\ref{eq:1}) satisfies the following conditions:
\begin{enumerate}
\item[1.]
the $n$-dimensional polytope $\mathcal A$ defined in (\ref{eq:A}) is not 
empty;
\item[2.]
 each objective in $\Q$ is bounded from below.
\end{enumerate}
\end{assumption}

\noindent
An immediate consequence of the first assumption is that
the projection $\mathcal Q=P\mathcal A$ is non-empty either,
and then $\Q^+=\Q+\O^p$ is full dimensional. According to Assumption 2,
$\Q$ and $\Q^+$ is contained in $y+\O^p$ for some (real)
vector $y\in\R^p$. 
Thus $\Q^+$ has exactly $p$ ideal vertices,
namely the positive endpoints 
of the coordinate axes, and these ideal vertices lie on the single ideal facet
of $\Q^+$.

\subsection{Benson's algorithm revisited}\label{subsec:alg1}

The {\em solution} of the \ref{eq:1} problem can be recovered from a 
description of the Pareto
front of $\Q^+$, which, in turn, is specified by the list of its
{\em vertices} and (non-ideal) {\em facets}. Indeed, the set of Pareto optimal points
is the union of those faces of $\Q^+$ which do not contain ideal points.
Thus solving \ref{eq:1} 
means {\em find all vertices and facets of the polytope $\Q^+$}.

\smallskip

Benson's ``outer approximation algorithm'' and its variants
\cite{loehne3,l-dual,loehne2,bensolve,loehne4} do exactly this, working in
the (low dimensional) objective space. These algorithms have two intertwined
parts: scalar optimization on one hand, and a combinatorial part on
the other. Their separation provides a deeper insight how these algorithms
work, and points toward a dual approach giving the title of this paper.

Benson's algorithm works in stages by maintaining a ``double description''
of an approximation of the final polytope $\Q^+$. A convex polytope is
uniquely determined as the convex hull of a set of points (for example, the
set of vertices) as well as the intersection of closed halfspaces
(for example, halfspaces corresponding to the facets). The ``double description''
refers to the technique keeping and maintaining both lists simultaneously.
The iterative algorithm stops when the last approximation equals $\Q^+$. At
this stage both the vertices and facets of $\Q^+$ are computed, thus the
\ref{eq:1} problem has been solved.

During the algorithm a new facet is added to the approximating polytope in
each iteration. The new facet is determined by solving a smartly chosen
scalar LP problem (specified from the description of the actual
approximation). Then the {\em combinatorial step} is executed: the new facet
is merged to the approximation by updating the facet and vertex lists. This
step is very similar to that of incremental vertex enumeration
\cite{avis-bremner,avis-fukuda,genov,zolotykh}, and can be parallelized.

Section \ref{sec:oracle} defines two types of black box algorithms which, 
on each call, provide data for the combinatorial part. The
{\em point separating oracle} separates a point from the
(implicitly defined) polytope $\Q^+$ by a halfspace. The {\em plane
separating oracle} is its dual: the input is a halfspace, and the oracle
provides a point of $\Q^+$ on the negative side of that halfspace.

Two general enumeration algorithms are specified in
Section \ref{sec:algorithms} which call these black box algorithms
repeatedly. 
It is proved that they terminate with a
description of their corresponding target polytopes. Choosing the initial
polytope and the oracle in some specific way, one gets
a convex hull algorithm, a (variant of) Benson's 
algorithm, the Ehrgott--L\"ohne--Shao dual algorithm,
and a new inner approximating algorithm.

Aspects of the combinatorial part are discussed in Section
\ref{sec:vertex-enum}.
Section \ref{sec:oracle-calls} explains how the oracles in these algorithms
can be realized. Finally, Section \ref{sec:remarks} discusses other
implementation details, limitations and possible generalizations.

\subsection{Our contribution}

The first skeletal algorithm in Section \ref{subsec:skeleton} is the
abstract versions of Benson's outer approximation algorithm and its
variants, where the combinatorial part (vertex enumeration) and the scalar
LP part has been separated. The latter one is modeled as an inquiry to a
{\em separating oracle} which, on each call, provides the data the
combinatorial part can work on. Using one of the the point
separation oracles in Section \ref{subsec:PSO}, one can recover, e.g.,
Algorithm 1 of \cite{l-dual}. The running time estimate in Theorem
\ref{thm:inner*} is the same (with the same proof) as the estimate in
\cite[Theorem 4.6]{l-dual}.

The other skeletal algorithm is the {\em dual} of the outer approximation
one. The ``inner'' algorithm of this paper uses the plane separation oracle
defined in  Section \ref{subsec:HSO}.
Another instance is Algorithm 2 of Ehrgott, L\"ohne and Shao
\cite{l-dual}, called ``dual variant of Benson's outer approximation
algorithm,''
which uses another weak plane separating oracle.
It would be interesting to see a more detailed description.

Studying these skeletal algorithms made possible to clarify the role of the
initial approximation, and prove general terminating conditions. While not
every separating oracle guarantees termination, it is not clear what are the
(interesting and realizable) sufficient conditions. Benson's outer
approximation algorithm is recovered by the first separation oracle defined in
Section \ref{subsec:PSO}. Other two separation oracles have stronger
termination guarantees, yielding the first version which is guaranteed to
take only the theoretically minimal number of iterations. Oracles in Section
\ref{subsec:HSO} are particularly efficient and contribute significantly
to the excellent performance of the new inner approximation algorithm.

It is almost trivial to turn any of the skeletal algorithms to an incremental
vertex (or facet) enumeration algorithm. Such algorithms have been studied
extensively. Typically there is a blow-up in
the size of the intermediate approximation; there are examples where even
the last but one approximation has size $\Omega(m^{\sqrt p/2})$ where $m$ is
the final number of facets (or vertices), and $p$ is the space dimension 
\cite{bremner}. This blow-up can be significant when $p$ is 6 or more.

An interesting extension to the usual single objective LP have been
identified, where not one but several goals are specified, and called 
``multi-goal LP.'' 
Optimization is done for the first goal, then {\em
within the solution space} the second goal is optimized, that is, in the
second optimization there is an additional constraint specifying that the
first goal has optimal value. Then within this solution space the third goal
is optimized, etc. Section \ref{subsec:multi-goal} sketches how existing LP
solvers can be patched to become an efficient multi-goal solver. We expect
more interesting applications for this type of solvers.

B\section{Separating oracles}\label{sec:oracle}

In this and in the following section $\B\subset\R^p$ is some {\em bounded},
closed, convex polytope with non-empty interior. As discussed in Section
\ref{subsec:ideal}, the condition that $\B$ is bounded can be replaced by
the weaker assumption that $\B$ does not contain a complete line.

A $p$-dimensional halfspace $H$ is specified by $\{y\in\R^p: yh\ge M\}$,
where the non-null vector $h\in\R^p$ is the normal, and the scalar $M$ is
the intercept. The {\em positive side} of $H$ is the open half-space
$\{y\in\R^p: yh > M\}$, the negative side is defined similarly. Facets of
the polytope $\B$ are identified with the halfspaces which contain $\B$ and
whose bounding hyperplane contains the facet.

\begin{definition}\label{def:PSO}\rm
A {\it point separating oracle} $\PSO(\B)$ for the polytope $\B\subset \R^p$
is a black box algorithm with the following input/output behavior:

input: a point $v\in\R^p$;

\hangafter=1\hangindent=1.5\parindent
output: ``inside'' if $v\in\B$; otherwise 
a halfspace $H$ corresponding to a facet of $\B$ such that 
$v$ is on the {\it negative} side of $H$.
\qed
\end{definition}

\begin{definition}\label{def:HSO}\rm
A {\it plane separating oracle} $\HSO(\B)$ for the polytope $\B\subset\R^p$
is a black box algorithm with the following input/output behavior:

input: a $p$-dimensional halfspace $H$;

\hangafter=1\hangindent=1.5\parindent
output: ``inside'' if $\B$ is a subset of $H$; otherwise
a vertex $v$ of $\B$ on the {\it negative side} of $H$.
\qed
\end{definition}

The main point is that only the oracle uses the polytope $\B$
in the enumeration algorithms of Section
\ref{sec:algorithms}, and $\B$ might not be defined
by a collection of linear constraints. In particular, this
happens when the algorithm is used to solve the multiobjective problem,
where the polytope passed to the oracle is $\Q$ or $\Q^+$.

The two oracles are dual in the sense that if $\B^*$ is the geometric dual
of the convex polytope $\B$ with the usual point versus hyperplane
correspondence \cite{ziegler}, then $\PSO(\B)$ and $\HSO(\B^*)$ are
equivalent: when a point is asked from oracle $\PSO(\B)$, ask the dual of
this point from $\HSO(\B^*)$, and return the dual of the answer.

\medskip

The object returned by a {\em weak} separating oracle is not required to
be a facet (or vertex), but only a supporting halfspace (or a boundary
point) separating the input from the polytope.
The returned separating object, however, cannot be arbitrary, must have some
structural property. Algorithms of Section \ref{sec:algorithms} work with
the weaker separating oracles defined below, but the performance guarantees
are exponentially 
worse. On the other hand such weak oracles can be realized easier.
Actually, strong separating oracles in Section \ref{sec:oracle-calls}
are implemented as tweaked weak oracles.

\begin{definition}\label{def:wPSO}\rm
A {\it weak point separating oracle} $\wPSO(\B$) for polytope $\B$
is a black box algorithm which works as follows. It fixes an
internal point $o\in\B$.

input: a point $v\in\R^p$;

\hangafter=1\hangindent=1.5\parindent
output: ``inside'' if $v\in\B$; otherwise connect $v$ to the fixed internal
point $o$, compute where this line segment intersects the boundary of $\B$,
and return a supporting halfspace $H$ of $\B$ whose boundary touches $\B$
at that point (this $H$ separates $v$ and $\B$).
\qed
\end{definition}

\begin{definition}\label{def:wHSO}\rm
A {\it weak plane separating oracle} $\wHSO(\B)$ for polytope $\B$ is a
black box algorithm which works as follows.

input: a halfspace $H$;

\hangafter=1\hangindent=1.5\parindent
output: ``inside'' if $\B$ is contained in $H$; otherwise
a boundary point $v$ of $\B$ on the negative side of $H$ 
which is farthest away from its bounding hyperplane.
\qed
\end{definition}

A strong oracle is not necessarily a weak one (for example, it can return 
any vertex on the negative side of $H$, not only the one which is farthest away
from it), but can always give a response which is consistent with being a
weak oracle. Similarly, there is always a valid response of a weak oracle
which qualifies as correct answer for the corresponding strong oracle.
 
While strong oracles are clearly dual of each other, it is not clear
whether the weak oracles are dual, and if yes, in what sense of duality.

\section{Approximating algorithms}\label{sec:algorithms}

The skeletal algorithms below are called {\em outer} and {\em inner}
approximations. The {\em outer} name reflects the fact that the algorithm
approximates the target polytope from outside, while the {\em inner} does it
from inside. The outer algorithm is an abstract version of Benson's
algorithm where the computational and combinatorial parts are separated.

\subsection{Skeletal algorithms}\label{subsec:skeleton}

On input both algorithms require two convex polytopes: $\S$ and $\B$. The
polytope $\S$ is the {\em initial approximation}; it is specified by
double description: by the list of its vertices and facets. The polytope
$\B$ is passed to the
oracle, and is specified according to the oracle's requirements. Both $\S$ 
and $\B$ are assumed to be closed, convex polytopes with non-empty interior. 
For this exposition they are also assumed to be {\em bounded}; 
this condition can
(and will) be relaxed to the weaker condition that none of them contains
a complete line.

\begin{algorithm}[Outer approximation]\label{alg:outer}\rm
Set $\S_0=\S$ as the initial approximation. In each approximating polytope
certain vertices will be marked as ``final.'' Initially this set is empty. 

Consider the $i$-th approximation $\S_i$. If all vertices of $\S_i$ are
marked final, then stop, the result is $\mathcal R=\S_i$. Otherwise pick a
non-final vertex $v_i\in\S_i$, and call the (weak) point separating oracle
with point $v_i$. If the oracle returns ``inside'', then mark $v_i$ as
``final'', and repeat. Otherwise the oracle returns (the equation of) a
halfspace $H$. Let $\S_{i+1}$ be the intersection of $\S_i$ and $H$.
Keep the ``final'' flag on previous vertices.
(Actually, all final vertices of $\S_i$ will be vertices of $\S_{i+1}$.)
Repeat.
\qed
\end{algorithm}

\begin{theorem}\label{thm:outer}
The outer approximation algorithm using the $\PSO(\B)$ oracle
terminates with the polytope $\mathcal
R=\B\cap\S$. The algorithm makes at most $v+f$ oracle calls, where $v$ is
number of vertices of $\mathcal R$, and $f$ is the number of facets of $\B$.
\end{theorem}
\begin{proof}
First we show that if the algorithm terminates then the result is
$\B\cap\S$. Each approximating polytope is an intersection of halfspaces
corresponding to certain facets of $\B$ (as $\PSO(\B)$ returns 
halfspaces which correspond to facets of $\B$), 
and all halfspace corresponding to facets of $\S_0$ thus $\B\cap\S\subseteq\S_i$.

If $\B\cap\S$ is a proper subset of $\S_i$, then $\S_i$ has a vertex $v_i$
not in $\B\cap\S$. This vertex $v_i$ cannot be marked ``final'' as final
vertices are always points of $\B\cap\S$. (All vertices of $\S_i$ are points
of the initial $\S$, and when the oracle returns ``inside'', the queried point is in
$\B$.) Thus the algorithm cannot stop at $\S_i$.

Second, the algorithm stops after making at most $v+f$ oracle calls. Indeed,
there are at most $v$ oracle calls which return ``inside'' (as a final
vertex is never asked from the oracle). Moreover, the oracle cannot return
the same facet $H$ of $\B$ twice. Suppose $H$ is returned at the $i$-th
step. Then $\S_{i+1}=\S_i\cap H$. If $j>i$ and $v_j$ is a vertex of
$\S_j\subseteq\S_{i+1}$, then $v_j$ is on the non-negative side of $H$,
i.e., the oracle cannot return the same $H$ for the query $v_j$.

From the discussion above it follows that vertices marked as ``final'' are
vertices of $\B\cap\S$, thus they are vertices of all subsequent
approximations. This justifies the sentence in parentheses in the
description of the algorithm.
\end{proof}

\begin{theorem}\label{thm:outer*}
The outer approximation algorithm using the weak $\wPSO(\B)$ oracle
terminates with the polytope $\mathcal R = \B\cap\S$.
The algorithm makes at most $v+2^f$ oracle calls, where $v$ is the number of
vertices of $\mathcal R$, and $f$ is the number of facets of $\B$.
\end{theorem}
\begin{proof}
Similarly to the previous proof, for each iteration we have 
$\B\cap\S\subseteq\S_i$, as each $H$ contains $\B$. If
$\B\cap\S$ is a proper subset of $\S_i$, then $\S_i$ has a vertex $v_i$ not
in $\B\cap\S$, thus not marked as ``final'', and the algorithm cannot stop
at this iteration.

To show that the algorithm eventually stops, we bound the number of oracle
calls. If for the query $v_i$ the response is ``inside'', then $v_i$ is a 
vertex of $\mathcal R$, and $v_i$ will not be asked again. This happens at
most as many times as many vertices $\mathcal R$ has.

Otherwise the oracle's response is a supporting halfspace $H_i$ whose
boundary touches $\B$ at a unique face $F_i$ of $\B$, which contains the
point where the segment $v_io$ intersects the boundary of $\B$.

If an internal point of the line segment $v_jo$ intersect the face $F_i$,
then $v_j$ must be on the negative side of $H_i$.
As $\S_{i+1}$ and all subsequent
approximations are on the non-negative side of $H_i$, if($j>i$ then $v_jo$
cannot intersect $F_i$, and then the face $F_j$ necessarily
differs from the face $F_i$. Thus there can be no more iterations than the
number of faces of $\B$. As each face is the intersection of some facets,
this number is at most $2^f$.
\end{proof}

Now we turn to the dual of outer approximation.

\begin{algorithm}[Inner approximation]\label{alg:inner}\rm
Set $\S_0=\S$ as the initial approximation. In each approximation certain
facets will be marked as ``final.'' Initially this set is empty.

Consider the $i$-th approximation $\S_i$. If all facets of $\S_i$ are
final, then stop, the result is $\mathcal R=\S_i$. Otherwise pick a
non-final facet $H_i$ of $\S_i$, and call the (weak) plane separating oracle
with the hyperplane $H_i$. If the oracle returns ``inside,'' then
mark $H_i$ as ``final,'' and repeat. Otherwise the oracle returns a boundary
point
$v_i$ of $\B$ on the negative side of $H_i$. Let $\S_{i+1}$ be the convex hull
of $\S_i$ and $v_i$. Keep the ``final'' flag on previous facets. (Actually,
all final facets of $\S_i$ will be facets of $\S_{i+1}$.) Repeat.
\qed
\end{algorithm}

\begin{theorem}\label{thm:inner}
The inner approximation algorithm using the $\HSO(\B)$ oracle terminates
with $\mathcal R =\conv(\B\cup\S)$, the convex hull of $\B$ and $\S$. The
algorithm makes at most $v+f$ oracle calls, where $v$ is the number of
vertices of $\B$, and $f$ is the number of facets of $\mathcal R$.
\end{theorem}

\begin{proof}
This algorithm is the dual of Algorithm \ref{alg:outer} above. The claims
can be proved along the proof of Theorem \ref{thm:outer} by replacing
notions by their dual: vertices by facets, intersection by convex hull,
calls to $\PSO$ by calls to $\HSO$, etc. Details are left to the
interested reader.
\end{proof}

\begin{theorem}\label{thm:inner*}
The inner approximation algorithm using the weak $\wHSO(\B)$ oracle
terminates with $\mathcal R =\conv(\B\cup\S)$.
The algorithm makes at most $2^v+f$ oracle calls, where $v$ is the
number of vertices of $\B$, and $f$ is the number of facets of $\mathcal R$.
\end{theorem}

\begin{proof}
As weak oracles are not dual, the proof is not as straightforward
as it was for the previous theorem. First, if the algorithm stops, it
computes the convex hull. Indeed, $\S_i\subseteq \mathcal R$, and if they
differ, then $\S_i$ has a facet with the corresponding halfspace $H$ 
and a vertex of $\B$ is on the negative side of $H$. 
This facet cannot be final, thus the algorithm did not stop.

To show termination, observe that at most $f$ oracle calls return
``inside''. In other calls the query was the halfspace $H_i$, and the
oracle returned $v_i$ from the boundary of $\B$ such that the distance
between $v_i$ and $H_i$ is maximal. Consider the face $F_i$ of $\B$ which
contains $v_i$ in its relative interior. 
We claim that the same $F_i$ cannot occur for any subsequent query.
Indeed, all points of $F_i$ are at exactly the same (negative) distance from
$H_i$, $v_i\in F_i$, $v_i$ is a point of $\S_{i+1}$ and all subsequent
approximations. If $j>i$ and $H_j$ is a facet of $\S_j$, then $v_i$ is on the
non-negative side of $H_j$. Consequently $v_i$ is not an element of the face
$F_j$, and then $F_i$ and $F_j$ differ.

Thus the number of such queries cannot be more than the number of faces in
$\B$, which is at most $2^v$.
\end{proof}

\subsection{A convex hull algorithm}\label{subsec:convex-hull}

The skeleton algorithms can be used with different oracles and initial
polytopes. An easy application is computing the convex hull of a point set
$V\subseteq \R^p$. In this case $\B=\conv(V)$, the convex hull of $V$. If
the initial approximation $\S$ is inside this convex hull, then the second
Algorithm \ref{alg:inner} just returns the convex hull $\B$. There is,
however, a problem. Algorithm \ref{alg:inner} uses the facet enumeration as
discussed in Section \ref{sec:vertex-enum}. This method generates the exact
facet list at each iteration, but keeps all vertices from the earlier
iterations, thus the vertex list can be redundant. When the algorithm stops,
the facet list of the last approximation gives the facets of the convex
hull. The last vertex list contains all of its vertices but might contain
additional points as well. A solution is to check all elements in the vertex
list by solving a scalar LP: is this element a a convex linear combination
of the others? Another option is to ensure that points returned by the
oracle are vertices of the convex hull -- namely using a strong oracle --,
and vertices of the initial polytope are not, thus they can be omitted
without any computation. This is what Algorithm \ref{alg:conv-hull} does.

\begin{algorithm}[Convex hull]\label{alg:conv-hull}\rm
On input a point set $V\subset \R^p$ construct the $p$-dimensional simplex
$\S$ as described in (\ref{alg:conv-hull}.1); then execute Algorithm
\ref{alg:inner} with the $\HSO$ oracle defined in (\ref{alg:conv-hull}.2).
Vertices of $\conv(V)$ are elements of the last vertex list minus vertices of
the initial polytope $\S$.
\begin{itemize}
\item[(\ref{alg:conv-hull}.1)]
Let $\bar v\in\R^p$ be the average of $V$, and choose $\eps$ be a small
positive number (say, smaller than all the positive distances 
$|v_i-\bar v_i|$ for all $v\in V$). Let $\S$ be the $p$ dimensional simplex
with vertices $\bar v$ and
$\bar v+\eps e_i$, where $e_i$ is the $i$-th coordinate vector. The $p+1$ facets
of $\S$ can be computed easily.
\item[(\ref{alg:conv-hull}.2)]
The input of the $\HSO$ oracle is the halfspace $H$, the output should be a
vertex of the convex hull. For each $v\in V$ compute the distance of $v$
from the halfspace $H$. If none of the distances is negative,
then return ``inside''. Otherwise let $V'$ be the set of all points where
the distances are minimal. The lexicographically minimal element of $V'$
will be a vertex of $\conv(V)$.
\qed
\end{itemize}
\end{algorithm} 

%

\subsection{Algorithms solving multiobjective linear
    optimization problems}\label{subsec:MOLP-algs}

Let us now turn to the problem of solving the multiobjective linear
optimization problem. As discussed in Section
\ref{subsec:alg1}, this means that we have to find a double description of
the (unbounded) polytope $\Q^+$. To achieve this, both skeletal algorithms
from Section \ref{subsec:skeleton}
can be used. Algorithm \ref{alg:benson} below is (a variant of) Benson's
original outer approximation algorithm \cite{l-dual}, while Algorithm
\ref{alg:minner}, announced in the title of this paper, uses inner
approximation. The second algorithm has several features which makes it
better suited for problems with large number (at least six) of objectives,
especially when $\Q^+$ has relatively few vertices.

As $\Q^+$ is unbounded, both algorithms below use ideal elements from the
extended ordered projective space $\widebar\R^p$ as elaborated in
\cite{projective}. Elements of $\widebar\R^p$ can be implemented using
homogeneous coordinates. For $1\le i\le p$ the ideal point at the positive
end of the $i$-th coordinate axis $e_i$ is denoted by $\e_i\in\widebar\R^p$.
For a (non-ideal) point $v\in \R^p$ the $p$-dimensional simplex with
vertices $v$ and $\{\e_i:1\le i\le p\}$ is denoted by $v+\E$. One facet of
$v+\E$ is the ideal hyperplane containing all ideal points; the other $p$
facets have equation $(y-v)e_i=0$, and the corresponding halfspaces are
$\{y\in\R^p: (y-v)e_i \ge 0\}$ as $v+\E$ is a subset of them The description
of how the oracles in Algorithms \ref{alg:benson} and \ref{alg:minner} are
implemented is postponed to Section \ref{sec:oracle-calls}.

The input of Algorithms \ref{alg:benson} and \ref{alg:minner} are the
$n$-dimensional polytope $\mathcal A$ as in (\ref{eq:A}), and the $p\times
n$ projection matrix $P$ defining the objectives, as in (\ref{eq:Q}). The
output of both algorithms is a double description of $\Q^+$. The first one
gives an exact list of its vertices (and a possibly redundant list of its
facets), while the second one generates a possibly redundant list of the
vertices, and an exact list of the facets.

\begin{algorithm}[Benson's outer algorithm]\label{alg:benson}\rm
Construct the initial polytope $\S\supseteq \Q^+$ as described in
(\ref{alg:benson}.1). Then execute Algorithm \ref{alg:outer} with the oracle
$\PSO(\Q^+)$, or $\wPSO(\Q^+)$, defined in Section \ref{subsec:PSO}.
\begin{itemize}
\item[(\ref{alg:benson}.1)]
Pick $v\in\R^p$ so that for all $1\le i\le p$, $v_i$ is a lower bound for
the $i$-th objective $P_i$, the $i$-th row of $P$. The initial outer
approximation $\S$ is the simplex $v+\E$. The coordinate $v_i$ can be found,
e.g., by solving the scalar LP 
$$
    \mbox{find } ~~ v_i = \min\nolimits_x\, \{ P_ix: x\in\mathcal A\}.
\dispqed$$
\end{itemize}
\end{algorithm}
By Theorems \ref{thm:outer} and \ref{thm:outer*} this algorithm terminates
with $\Q^+$, as required. When creating the initial polytope $\S$ in step
\ref{alg:benson}.1, the scalar LP must have a feasible solution (otherwise
$\mathcal A$ is empty), and should not be unbounded (as otherwise the $i$-th
objective is unbounded from below).
After $\mathcal S$ has been created successfully we know that
its ideal vertices are also vertices of $\Q^+$, thus they can be marked as
``final''. As these are the only ideal vertices of $\Q^+$, no further ideal
point will be asked from the oracle at all. These facts can be used to
simplify the oracle implementation.

Recall that $v+\E$ is the $p$-dimensional simplex whose vertices are $v$ and
the ideal points at the positive endpoints of the coordinate axes.

\begin{algorithm}[Inner algorithm]\label{alg:minner}\rm
Construct the initial approximation $\S$ as described in
(\ref{alg:minner}.1). Then execute Algorithm \ref{alg:inner} with the oracle
$\HSO(\Q)$, or $\wHSO(\Q)$ described in Section \ref{subsec:HSO}.
\begin{itemize}
\item[(\ref{alg:minner}.1)]
Pick a point $v\in\Q$; the initial polytope $\S$ is the simplex
$v+\E$. This $v$ can be found as $v=Px\in\R^p$ for any
$x\in\mathcal A$ in (\ref{eq:A}); or $v$ could be the answer of the
oracle to the halfspace query $\{ y\in\R^p: ye \ge M\}$,
where $e\in\R^p$ is the all-one vector and $M\in\R$ is a very large scalar.
\qed
\end{itemize}
\end{algorithm}

Theorems \ref{thm:inner} and \ref{thm:inner*} claim that this algorithm
computes a double description of $\Q^+$, as the convex hull of $\Q$ and
$\S$ is just $\Q^+$. Observe that in this case the oracle answers questions
about the polytope $\Q$, and not about $\Q^+$.
The ideal facet of $\S$ is also a facet
of $\Q^+$, thus it can be marked ``final'' and then it won't be asked 
from the oracle. No ideal point should ever be returned by the oracle.

To simplify the generation of the initial approximation $\S$, the oracle may
accept any normal vector $h\in\R^p$ as a query, and return a (non-ideal) vertex (or
boundary point) $v$ of $\Q$ which minimizes the scalar product $hv$, and
leave it to the caller to decide whether $\Q$ is on the
non-negative side of the halfspace $\{y\in\R^p: yh\ge M\}$. This happens
when $vh\ge M$ for the returned point $v$.
%

If the initial approximation $\S$ is $v+\E$ for some vertex $v$ of $\Q$ (as
indicated above) and
the algorithm uses the strong oracle $\HSO(\Q$, then vertices of all
approximations are among the vertices of $\Q^+$. Consequently the final
vertex list does not contain redundant elements. In case of using a weak 
plane separating oracle or a different initial approximation, the final
vertex list might contain additional points. To get an exact solution the
redundant points must be filtered out.

\section{Vertex and facet enumeration}\label{sec:vertex-enum}

This section discusses the combinatorial part of the skeleton Algorithms
\ref{alg:outer} and \ref{alg:inner} in some detail. For a more throughout
exposition of this topic please consult
\cite{avis-bremner,avis-fukuda,fukuda-prodon}. At each step of these
algorithms the approximation $\S_i$ is updated by adding a new halfspace
(outer algorithm) or a new vertex (inner algorithm). To maintain the double
description, we need to update both the list of vertices and the list of
halfspaces. In the first case the new halfspace is added to the list
(creating a possibly redundant list of halfspaces), but the vertex list
might change significantly. In the second case it is the other way around:
the vertex list grows by one (allowing vertices which are inside the convex
hull of the others), and the facet (halfspace) list changes substantially.
As adding a halfspace and updating the exact vertex list is at the heart of
incremental vertex enumeration \cite{genov,zolotykh}, it will be discussed
first.

Suppose $\S_i$ is to be intersected by the (new) halfspace $H$. Vertices of
$\S_i$ can be partitioned as $V^+\cup V^-\cup V^0$: those which are on the
positive side, on the negative side, and those which are on the boundary of 
$H$. As the
algorithm proceeds, $H$ always cuts $\S_i$ properly, thus neither $V^+$ nor
$V^-$ is empty; however $V^0$ can be empty. Vertices in $V^+$ and in $V^0$
remain vertices of $\S_{i+1}$; vertices in $V^-$ will {\rm not} be vertices
of $\S_{i+1}$ any more, and should be discarded. The new vertices of
$\S_{i+1}$ are the points where the boundary of $H$ intersects the edges 
of $\S_i$ in some 
internal point. Such an edge must have one endpoint in $V^+$, and the other
endpoint in $V^-$. To compute the new vertices, for each pair $v_1v_2$ with
$v_1\in V^+$ and $v_2\in V^-$ we must decide whether $v_1v_2$ is and edge of
$S_i$ or not. This can be done using the well-known and popular necessary
and sufficient combinatorial test stated in Proposition \ref{edge-test}$\,$b)
below.
As executing this test is really time consuming, the faster necessary 
condition a) is checked first which filters out numerous non-edges.
A proof for the
correctness of the tests can be found, e.g., in \cite{fukuda-prodon}.

\begin{proposition}\label{edge-test}
{\rm a)} If $v_1v_2$ is an edge, then there must be at least $p-1$ different
facets containing both $v_1$ and $v_2$.

{\rm b)} $v_1v_2$ is an edge if and only if for every other vertex
$v_3$ 
there is a facet $H$ which contains $v_1$ and $v_2$ but not $v_3$.
\qed
\end{proposition}

Observe that Proposition \ref{edge-test} remains valid if the set of facets
are replaced by the boundary hyperplanes of any collection of halfspaces 
whose intersection is $\S_i$. That is, we can have ``redundant'' elements in
the facet list. We need not, and will not, check whether adding a new facet
makes other facets ``redundant.''

Both conditions can be checked using the vertex--facet adjacency matrix,
which must also be updated in each iteration. In practice the adjacency matrix
is stored
twice: using the bit vector $\mathbf H_v$ for each vertex $v$ (indexed by
the facets), and the bit vector $\mathbf V_H$ for each facet $H$ (indexed by
the vertices). The vector $\mathbf H_v$ at position $H$
contains 1 if and only if $v$ is on the facet $H$; similarly for the
vector $\mathbf V_H$.
Condition a) can be expressed as the intersection $\mathbf H_{v_1}\land
\mathbf H_{v_2}$ has at least $p-1$ ones. Condition b) is the same as
the bit vector
$$
     \bigwedge \big\{ \mathbf V_H: \, H\in \mathbf H_{v_1}\land \mathbf H_{v_2}
    \big\}
$$
contains only zeros except for positions $v_1$ and $v_2$. Bit operations are
quite fast and is done on several (typically 32 or 64) bits
simultaneously. Checking which vertex pairs in $V^+\times V^-$ are edges, and
when such a pair is an edge computing the coordinates of the new vertex, 
can be done in parallel. Almost the complete execution time of vertex
enumeration is taken by this edge checking procedure. Using multiple threads
the total work can be evenly distributed among the threads with very little
overhead, practically dividing the execution time by the number of the 
available threads. See Section \ref{subsec:parallel} for further remarks.

Along updating the vertex and facet lists, the adjacency vectors must be
updated as well. In each step we have one more facet, thus vectors $\mathbf
H_v$ grow by one bit. The size of facet adjacency vectors change more
erratically. We have chosen a lazy update heuristics: there is an upper
limit on this size. Until the list size does not exceeds this limit, newly
added vertices are assigned a new slot, and positions corresponding to
deleted vertices are put on a ``spare indices'' pool. When no more new slots
are available, the whole list is compressed, and the upper limit is expanded
if necessary.

\medskip

In Inner Algorithm \ref{alg:inner} the sequence of intermediate polytopes
$\S_i$ grows by adding a new vertex rather than cutting it by a new facet.
In this case the facet list is to be updated. As this is the dual of 
the procedure above, only the main points are highlighted. In this case the
vertex list can be redundant, meaning some of them can be inside the convex
hull of others, but the facet list must contain no redundant elements.

Let $v$ be the new vertex to be added to the polytope $\S_i$. Partition the
facets of $\S_i$ as $H^+\cup H^-\cup H^0$ such that $v$ is on the positive
side of facets in $H^+$, on the negative side of facets in $H^-$, and is on
hyperplane defined by the facets in $H^0$. The new facets of $\S_{i+1}$ are 
facets in $H^+$ and
$H^0$, plus facets determined by $v$ and a ridge (a $(p-2)$-dimensional
face) which is an intersection of one facet from $H^+$ and one facet from
$H^-$. One can use the dual of Proposition \ref{edge-test} to check whether
the intersection of two facets is a ridge or not.
 We state this condition
explicitly as this dual version seems not to be known.

\begin{proposition}\label{ridge-test}
{\rm a)} If the intersection of facets $H_1$ and $H_2$ is a ridge, then
they share at least $p-1$ vertices in common.

{\rm b)} $H_1\cap H_2$ is a ridge if and only if for every other facet $H_3$
there is a vertex $v\in H_1\cap H_2$ which is not in $H_3$.
\qed
\end{proposition}
%

\section{Realizing oracle calls}\label{sec:oracle-calls}

This section describes how the oracles in Algorithms \ref{alg:benson} and
\ref{alg:minner} can be realized. In both cases a weak separating oracle is
defined first, and the we show how can it be tweaked to become a strong
oracle. The solutions are not completely satisfactory. Either the hyperplane
returned by the point separating oracle for Benson's algorithm
\ref{alg:benson} will be a facet with ``high probability'' only, or the
oracle requires a non-standard feature from the underlying scalar LP solver.
Fortunately weak oracles work as well, thus the improbable but possible
failure of the ``probability'' oracle is not fatal: it might add further
iterations but neither the correctness nor termination is affected.
(Numerical stability is another issue.) In Section \ref{subsec:multi-goal}
we describe how a special class of scalar LP solvers -- including the GLPK
solver \cite{glpk} used in the implementation of {\em Bensolve}
\cite{bensolve} and {\em Inner} -- can be patched to find the
lexicographically minimal element in the solution space.

The polytopes on which the oracles work are defined by the $m\times n$ and 
$p\times n$ 
matrices $A$ and $P$, respectively, and by the vector $c\in\R^m$ as follows:
$$\begin{array}{l@{\:=\:}l}
   \mathcal A & \{ x\in\R^n:\, Ax=c, \,~~ x \ge 0 \,\}, \\[3pt]
   \Q & P\mathcal A = \{ Px \,: x \in\mathcal A\,\}, \\[3pt]
   \Q^+ & \Q+\O^p = \{ y+z\,: y\in\Q, ~ z \in\O^p\,\}.
\end{array}$$
The next two sections describe how the weak and strong oracles required by
Algorithms \ref{alg:benson} and \ref{alg:minner} can be realized.

\subsection{Point separation oracle}\label{subsec:PSO}

The oracle's input is a point $v\in\R^p$, and the response should be a
supporting halfspace (weak oracle) or a facet (strong oracle) of $\Q^+$
which separates $v$ and $\Q^+$. As discussed in Section
\ref{subsec:MOLP-algs}, several simplifying assumptions can be 
made, namely
\begin{itemize}\setlength\itemsep{2pt}
\item $\Q$ is not empty and bounded from below in each objective direction;
\item $v$ is not an ideal point;
\item if $v$ is in $\Q^+$, then it is a boundary point.
\end{itemize}

Let us consider first the weak oracle. According to Definition
\ref{def:wPSO} the $\wPSO(\Q^+)$ must choose a fixed internal point
$o \in\Q^+$. The
vertices of the ideal facet of $\Q^+$ are the positive endpoints of the
coordinate axes. If all coordinates of $e^*\in\R^p$ are positive, then the
ideal point corresponding to this vector is in the relative interior of the
ideal facet. Fix this vector $e^*$, and let $o$ be the ideal point
corresponding to $e^*$.
While $o$ is not an internal point of $\Q^+$, this choice works as no
ideal point is ever asked from the oracle according to the assumptions above.
Thus let us fix $o$ this way.

Given the query point $v$, the ``segment'' $vo$ is the ray $\{v-\lambda e^*:
\lambda<0\}$. The $vo$ line intersects the boundary of $\Q^+$ at the point
$\hat v = v - \hat\lambda e^*$, where $\hat\lambda$ is the solution of the
scalar LP problem
\begin{equation}[P$(v,e^*)$]\label{eq:wPSO}
\hat\lambda =  \max\nolimits_{\lambda,x}\,\left\{
    \lambda:\, Px \le v - \lambda e^*,\, Ax=c,\,x\ge 0 \right\}.
\end{equation}
Indeed, this problem just searches for the largest $\lambda$ for which
$v-\lambda e^* \in \Q^+$. By the assumptions above \ref{eq:wPSO} always
has an optimal solution.

Now the line $vo$ intersects $\Q^+$ in the ray $\hat vo$. Thus
$v\in\Q^+$ if and only if $\hat\lambda\ge 0$; in this case the oracle returns
``inside.'' In the $\hat\lambda<0$ case the oracle should
find a supporting halfspace $H$ to $\Q^+$ at the boundary point
$\hat v$. Interestingly, the normal of these supporting hyperplanes
can be read off from the solution space
of the {\em dual} of \ref{eq:wPSO}, where the dual
variables are $s\in\R^p$ and $t\in\R^m$:
\begin{equation}[D$(v,e^*)$]\label{eq:wPSO-dual}
   \min\nolimits_{s,t}\,\left\{
      s^Tv + t^Tc:\, s^TP + t^TA \ge 0,\, s^Te^* = 1,\, s\ge 0
   \right\}.
\end{equation}
As the primal \ref{eq:wPSO} has an optimal solution, the
strong duality theorem says that the dual \ref{eq:wPSO-dual} has the same
optimum $\hat\lambda$.

\begin{proposition}\label{prop:1}
The problem \ref{eq:wPSO-dual} takes its optimal value $\hat\lambda$
at $s\in\R^p$  (together with some $t\in\R^m)$ if and only if
$s$ is a normal of a supporting halfspace to $\Q^+$ at $\hat v$.
\end{proposition}

\begin{proof}
By definition $s$ is a normal of a supporting halfspace to $\Q^+$ at 
$\hat v$ if and only if $s(y-\hat v)\ge 0$ for all $y\in\Q^+$. From here it
follows that $s\ge 0$ (as otherwise $sy$ is not bounded from below for
some large enough $y\in\Q^+$), and as $e^*$ has all positive coordinates,
$s$ can be normalized by assuming $s^Te^*=1$. 

Knowing that $s\ge 0$, if $sy\ge s\hat v$ holds for some $y\in\R^p$, then it
also holds for every $y'\ge y$. Thus it is enough to require $sy\ge s\hat v$
for all $y\in\Q$ only, that is,
\begin{equation}\label{eq:prop-1}
    s^T(Px) \ge s^T\hat v ~~~~\mbox{ for all $x\in\mathcal A$}.
\end{equation}
The polytope $\mathcal A$ is non-empty. According to the strong duality
theorem, all $x\in\mathcal A$ satisfies a linear inequality if and only if 
it is a linear combination of the defining (in)equalities of $\mathcal A$.
Thus (\ref{eq:prop-1}) holds if and only if there is a vector $(-t)\in\R^m$
such that
$$
    s^TP \ge -t^TA , ~~~\mbox{ and } ~~~ s^T \hat v = -t^T c .
$$
Plugging in $\hat v=v-\hat\lambda e^*$ and using $s^Te^*=1$ we get that $s$
is a normal of a supporting halfspace of $\Q^+$ at $\hat v$ if and only if
$$
   s^TP + t^TA \ge 0, ~~~\mbox{ and } ~~~ s^Tv + t^T c = \hat\lambda,
$$
namely $(s,t)$ is in the solution space of \ref{eq:wPSO-dual}.
\end{proof}

Proposition \ref{prop:1} indicates immediately how to define a weak
separating oracle.

\begin{oracle}[weak $\wPSO(\Q^+)$]\label{def:oracle-1}\rm 
Fix the vector $e^*\in\R^p$ with all positive coordinates.
On input $v\in\R^p$, solve
\ref{eq:wPSO-dual}. Let the minimum be $\hat \lambda$, and the place where
the minimum is attained be $(\hat s,\hat t)$. If $\hat\lambda\ge 0$, then
return ``inside''. Otherwise let $\hat v = v - \hat\lambda e^*$, and the
supporting halfspace to $\Q^+$ at $\hat v$ is $\{ y\in\R^P: \hat sy\ge\hat
s\hat v \}$.
\qed
\end{oracle}

The oracle works with any positive $e^*$.
Choosing $e^*$ to be the all one vector is a possibility which
has been made by {\em Bensolve} and other variants. Choosing other vectors
$e^*$ can be advantageous.
Observe that the supporting plane at the boundary point $\hat v$ is a facet 
if it is is {\em not} in any face of dimension $(p-2)$ or less. Given the point
$v\notin\Q^+$ one would like to avoid directions $e^*$ for which $v-\lambda e^*$ 
hits $\Q^+$ in such a low-dimensional face. The set of these bad directions
has measure zero (in the usual Lebesgue sense), so we expect that choosing
$e^*$ ``randomly'' the returned halfspace will be a facet with high probability.
This heuristic argument works quite well in practice.

\begin{oracle}[probabilistic $\wPSO(\Q^+)$]\label
{def:oracle-2}\rm
Choose $e^*$ randomly according to the uniform distribution from the 
$p$-dimensional cube $[1,2]^p$.
Then execute Oracle \ref{def:oracle-1} using this random vector $e^*$.
\qed
\end{oracle}

As there is no guarantee that Oracle \ref{def:oracle-2} always 
returns a facet, this is only a weak separating oracle.

\medskip

Proposition \ref{prop:1} suggests another way to extract a facet among the
supporting hyperplanes. The {\em solution space} of \ref{eq:wPSO-dual} in
the first $p$ variables spans the (convex polyhedral) space of the {\em
normals} of the supporting hyperplanes. Facet normals are the {\em extremes}
among them. Pinpointing a single extreme among them is easy. As in the case
of the convex hull algorithm in Section \ref{subsec:convex-hull}, choose its
{\em lexicographically minimal} element: this will be the normal of a facet.

\begin{oracle}[strong $\PSO(\Q^+)$]\label{def:oracle-3}\rm
Fix the vector $e^*\in\R^p$ with all positive coordinates, e.g., take the
all one vector. On input
$v\in\R^p$ solve \ref{eq:wPSO-dual}. The LP solver must return the minimum,
and the {\it lexicographically minimal} $s\in\R^p$ from the solution space.
Proceed as in Oracle \ref{def:oracle-1}.
\qed
\end{oracle}

The question how to find the lexicographically minimal element in the
solution space is addressed in the next section.

\subsection{A ``multi-goal'' scalar LP solver}\label{subsec:multi-goal}

Scalar LP solvers typically stop when the optimum value is reached, and
report the optimum, the value of the structural variables,
and optionally the value of the dual variables. The lexicographically minimal
vector from the solution space can be computed by solving additional
(scalar) LP problems by adding more and more constraints which restrict the
solution space. For Oracle \ref{def:oracle-3} this sequence could be
$$\begin{array}{l@{\:=\:}l}
   \hat\lambda & \min\nolimits_{s,t}\,\{
      s^Tv+t^Tc:\, \mbox{original constraints} ~
   \},\\[3pt]
   \hat s_1 & \min\nolimits_{s,t}\,\{ s_1:\,
      s^Tv+t^Tc=\hat\lambda ~ + \mbox{original constraints} ~ \}, \\[3pt]
   \hat s_2 & \min\nolimits_{s,t}\,\{ s_2:\,
      s_1=\hat s_1,~ s^Tv+t^Tc=\hat\lambda +{}\\
   \multicolumn{2}{r}{~~~~~ + \mbox{original constraints} ~ \},} \\[3pt]
   \hat s_3 & \min\nolimits_{s,t}\,\{ s_3:\,
      s_2=\hat s_2,~ s_1=\hat s_1,~ s^Tv+t^Tc=\hat\lambda+{}\\
   \multicolumn{2}{r}{~~~~~ + \mbox{original constraints} ~ \},} \\[3pt]
   \multicolumn{2}{l}{\mbox{etc.}}
\end{array}$$
The first LP solves \ref{eq:wPSO-dual}. The second one fixes this solution
and searches for the minimal $s_1$ within the solution space, etc.
Fortunately, certain simplex LP solvers -- including GLPK \cite{glpk} -- can
be patched to do this task more efficiently. Let us first define what do we
mean by a ``multi-goal'' LP solver, show how it can be used to realize the
strong Oracle \ref{def:oracle-3}, and then indicate an efficient
implementation.

\begin{definition}\label{def:multi-goal-LP}\rm
A {\it multi-goal} LP solver accepts as input a set of linear constraints
and several goal vectors, and 
returns a feasible solution after minimizing for all goals in sequence:
\begin{trivlist}\setlength\itemsep{2pt}
\item[\quad]
\textbf{Constraints}: a collection of linear constraints.

\item[\quad]
\textbf{Goals}: real vectors $g_1$, $g_2$, $\dots$, $g_k$.
\end{trivlist}
Let $\mathcal A_0=\{x: x$ satisfies the given constraints $\}$; and for
$1\le j\le k$ let $\mathcal A_j=\{x\in\mathcal A_{j-1}: g_jx$ is minimal
among $x\in\mathcal A_{j-1}\}$.

\begin{trivlist}\setlength\itemsep{2pt}
\item[\quad]
\hangindent=2\parindent\hangafter=1
\textbf{Result}: 
Return any vector $x$ from the last set $\mathcal A_k$, or report failure if
any of $\mathcal A_j$ is empty (including the case when $g_jx$ is not
bounded from below).
\qed
\end{trivlist}
\end{definition}

\noindent
Oracle \ref{def:oracle-3} can be realized by such a multi-goal LP
solver. Fix $e^*\in\R^p$. We need the lexicographically minimal $s\in\R^p$
from the solution space of \ref{eq:wPSO-dual}. 
Let $e_i\in\R^{p+m}$ be the coordinate vector with 1 in position $i$.
Call the multi-goal LP solver as follows:
\begin{trivlist}\setlength\itemsep{2pt}
\item[\quad]\textbf{Constraints}: $s^TP + t^TA \ge 0$, $s^Te^*=1$, $s\ge 0$.
\item[\quad]\textbf{Goals}: the $p+m$ dimensional vectors $(v,c)$, $e_1$, $e_2$, 
  $\dots$, $e_p$.
\end{trivlist}
Parse the $p+m$ dimensional output as $(\hat s, \hat t)$, and compute
$\hat\lambda=\hat s^Tv+\hat t^Tc$, $\hat v = v-\hat\lambda e^*$. Return ``inside'' if
$\hat\lambda\ge 0$, and return the halfspace equation
$\{ y\in\R^P: \hat sy \ge \hat s\hat v\}$ otherwise.

\medskip

In the rest of the section we indicate how GLPK \cite{glpk} can
be patched to become an efficient multi-goal LP solver. After 
parsing the constraints, GLPK internally transforms them to the form
$$
              Ax = c, ~~~ a\le x \le b.
$$
Here $x\in\R^n$ is the vector of all variables, $A$ is some $m\times n$ matrix,
$m\le n$, $c\in\R^m$ is a constant vector, and $a,b\in\R^m$ are the {\em
lower} and {\em upper bounds}, respectively. The lower bound $a_i$ can be
$-\infty$ and the upper bound $b_i$ can be $+\infty$ meaning that the
variable $x_i$ is not bounded in that direction. The two bounds can be
equal, but $a_i\le b_i$ must always hold (otherwise the problem has no
feasible solution). The function to be minimized is $gx$ for some fixed goal
vector $g\in\R^n$. (During this preparation only slack and auxiliary
variables are added, thus $g$ can be got from the original goal vector by
padding it with zeroes.)

GLPK implements the simplex method by maintaining a set of $m$ {\em base
variables}, an invertible $m\times m$ matrix $M$ so that the product $MA$
contains (the permutation of) a unit matrix at column positions
corresponding to the base variables. Furthermore, for each non-base variable
$x_i$ there is a flag indicating whether it takes its lower bound $a_i$ or
upper bound $b_i$. Fixing the value of non-base variables according
to these flags, the value of base variables can be computed from the equation
$MA x = Mc$. If the values of the computed base variables fall 
between the corresponding
lower and upper bounds, then the arrangement (base, $M$, flags) is
{\em primal feasible}.\footnote{Observe that the constraint matrix $A$ 
never changes. GLPK stores only non-zero elements of $A$ in doubly linked 
lists, which is quite space-efficient when $A$ is sparse. All computations
involving $A$ are made ``on the fly'' using this sparse storage.}

The simplex method works by changing one feasible arrangement to another one
while improving the goal $gx$. Suppose we are at a feasible arrangement where
$gx$ cannot decrease.
Let $\phi\in\R^m$ be a vector such that $g' = g - \phi^T (M A)$
contains zeros at indices corresponding to base variables. As $MA$ contains
a unit matrix at these positions, computing this $\phi$ is trivial. Now $g'x
= gx - \phi^T(MA)x = gx - \phi^TMc$, thus $g'x$ is minimal as well. By
simple inspection $g'x$ cannot be decreased if and only if for each 
non-base variable $x_i$ one of the following three possibilities hold:
\begin{itemize}\setlength\itemsep{2pt}
\item[] $g'_i=0$;
\item[] $g'_i > 0$ and $x_i$ is on its upper limit;
\item[] $g'_i < 0$ and $x_i$ is on its lower limit.
\end{itemize}
When this condition is met, the optimum is reached. At this point, instead of
terminating the solver, let us
change the bounds $a_i\le b_i$ to $a'_i\le b'_i$ as follows. For base
variables $x_i$ and when $g'_i=0$ keep the old bounds: $a'_i=a_i$,
$b'_i=b_i$. If $g'_i>0$ 
then set $a'_i=b'_i=b_i$ (shrink it to the upper limit), and if $g'_i<0$ then
set $a'_i=b'_i=a_i$ (shrink it to the lower limit).

\begin{proposition}\label{prop:2}
The solution space of the scalar LP
$$
   \min\nolimits_x\,\big\{ gx \,:\, Ax = c, ~ a\le x\le b \,\big\}
$$
is the collection of all $x\in\R^n$ which satisfy
$$
     \big\{ Ax=c, ~ a'\le x \le b' \big\}.
$$
\end{proposition}
\begin{proof}
From the discussion above it is clear that if $gx$ is optimal
then $x$ must satisfy $a'\le x \le b'$, otherwise $gx$ can
decrease.
\end{proof}

This indicates how GLPK can be patched.
When the optimum is reached, change the bounds $(a,b)$ to $(a',b')$.
After this change all feasible
solutions are optimal solutions for the first goal. Change the goal function
to the next one and resume execution. The last arrangement remains feasible,
thus there is no overhead, no additional work is required, just continue
the usual simplex iteration. 
Experience shows that each additional goal adds only a few or no
iterations at all. The performance loss over returning after the first
minimization is really small.

\subsection{Plane separation oracle}\label{subsec:HSO}

The plane separation oracles used in Algorithm \ref{alg:minner}
can be implemented as indicated at the end of Section
\ref{subsec:MOLP-algs}. The oracle's input is a vector $h\in\R^m$ and an
intercept $M\in\R$. The
oracle should return a point $v\in\R^m$ on the boundary of $\Q$ (weak
oracle), or a vertex of $\Q$ (strong oracle) for which the scalar product
$hv$ is the smallest possible. Both the oracle (and the procedure calling
the oracle) should be prepared for cases when $\Q$ is empty, or when $\{ hv
: v\in\Q\}$ is not bounded from below.

\begin{oracle}[weak $\wHSO(\Q)$]\label{def:oracle-wHSO}\rm
On input $h\in\R^p$ and intercept $M\in\R$ find any $\hat x\in\R^m$ 
where the scalar LP problem
\begin{equation}[P$(h)$]\label{eq:owHSO}
    \min\nolimits_x\,\{ h^TPx, ~ Ax=c,~ x\ge 0 ~\}
\end{equation}
takes its minimum. If the LP fails, return failure. Otherwise compute
$v=P\hat x\in\R^p$. Return ``inside'' if $M$ is specified and $hv\ge M$,
otherwise return $v$.
\qed
\end{oracle}

The scalar LP problem \ref{eq:owHSO} searches for a point of $\Q$ which is farthest
away in the direction of $h$. As all oracle calls specify a normal $h\ge 0$,
the LP can fail only when some of the assumptions on the MOLP problem
in Section \ref{subsec:assumptions} does not hold. In such a case the MOLP
solver can report the problem and abort.

The very first oracle call can be used to find a boundary point of
$\Q$ Algorithm \ref{alg:minner} starts with. In this case the oracle is
called without the intercept. This first oracle call will complain
when the the polytope $\mathcal A$ is empty, but might not detect
if $\Q$ is not bounded from below in some objective direction. The input to
subsequent oracle calls is the facet equation of the approximating polytope,
thus have both normal and intercept.

\medskip

To guarantee that the oracle returns a vertex of $\Q$, and not an
arbitrary point on its boundary at maximal distance, the multi-goal scalar LP
solver from Section \ref{subsec:multi-goal} can be invoked.

\begin{oracle}[strong $\HSO(\Q)$]\label{def:oracle-HSO}\rm
On input $h\in\R^p$ and intercept $M\in\R^p$ call the multi-goal LP solver
as follows:
\begin{trivlist}\setlength\itemsep{2pt}
\item[\quad]\textbf{Constraints}:
   $Ax=c$, $x\ge 0$.
\item[\quad]\textbf{Goals}:
   $h^TPx$, $P_1x$, $P_2x$, $\dots$, $P_px$,
\end{trivlist}
where $P_i$ is the $i$-the row of the objective matrix $P$. The solver
returns $\hat x$. Compute $\hat v = P\hat x$. Return ``inside'' if $M$ is
specified and $h\hat v\ge M$, otherwise return $\hat v$.
\qed
\end{oracle}

The point $\hat v$ is lexicographically minimal among the boundary points of
$\Q$
which are farthest away from the specified hyperplane, thus -- assuming the
oracle did not fail -- it is a vertex.

Observe that the plane separating oracles in this section work directly on
the polytope $\mathcal A$ without modifying or adding any further
constraints. Thus the constraints and all but the first goal in the
invocations of the multi-goal LP solver are the same.

\section{Remarks}\label{sec:remarks}

\subsection{Using different ordering cone}\label{subsec:ordering-cone}

In a more general setting the minimization in the MOLP problem
\begin{equation}[MOLP]
   \mbox{find } ~~\min\nolimits_y\,\{y:\,y\in\Q\,\}
\end{equation}
is understood with respect to the partial ordering on $\R^p$ defined by the
(convex, closed, pointed, and polyhedral) {\em ordering cone} $\mathcal C$. In
this case the solution of the MOLP problem is the list of facets and
vertices of the Minkowski sum $\Q^+_{\mathcal C}=\Q+\mathcal C$.
The MOLP problem is {\em $\mathcal C$-bounded} just in case the ideal points of
$\Q^+_{\mathcal C}$ are the extremal rays of $\mathcal C$. In this case the
inner algorithm works with the same plane separating oracles
$\HSO(\Q)$ and $\wHSO(\Q)$ as defined in Section \ref{subsec:HSO} 
whenever the vertices of the initial approximation $\mathcal S$ are just these
extremal ideal points and some point $v\in\Q$. Indeed, in this case we have 
$\Q^+_{\mathcal C} = \conv(\Q\cup \S)$, and the algorithm terminates with a
double description of this polytope.

When using the outer approximation algorithm, the point separating oracle is
invoked with the polytope $\Q^+_{\mathcal C}$. The vector $e^*$ should be
chosen to be an internal ray in $\mathcal C$, and the scalar LP searching
for the intersection of $v-\lambda e^*$ and the boundary of $\Q^+_{\mathcal
C}$ is
$$
   \hat\lambda = \max\nolimits_{\lambda,x,z}\,
     \{ \lambda: v-\lambda e^* = Px + z,~ Ax=c, ~ x\ge 0,~ z\in\mathcal C\},
$$
see Section \ref{subsec:PSO}. The analog of Proposition \ref{prop:1}
can be used to realize the oracles $\PSO(\Q^+_{\mathcal C})$ and
$\wPSO(\Q^+_{\mathcal C})$. The base of the initial bounding $\S$ is
formed by the ideal vertices at the end of the extremal rays of $\mathcal
C$, as above. The top vertex can be
generated by computing a $H$-minimal point of $\Q$ for each facet $H$
of $\mathcal C$.

\subsection{Relaxing the condition on boundedness}

When the boundedness assumption in Section \ref{subsec:assumptions} does not
hold, then the ideal points of $\Q^+$ (or the polytope $\Q^+_{\mathcal C}$
above) are not known in advance. As the initial approximation of the Outer
algorithm \ref{alg:benson} must contain $\Q^+$, these ideal points should be
computed first.
For Algorithm \ref{alg:minner} there are two possibilities. One
can start from the same initial polytope $\S$, but then the oracle can
return ideal points.
Or, as above,
the initial approximation could contain all ideal points, and then all
points returned by the oracle (which could just be $\HSO(\Q)$ or
$\wHSO(\Q)$) are non-ideal points.

The ideal points of $\Q^+$ (or $\Q^+_{\mathcal C}$) are the convex hull of
the ideal points of $\Q$ and that of the positive orthant $\O^p$ (or the
ordering cone $\mathcal C$). To find all the ideal vertices, Algorithm
\ref{alg:minner} can be called recursively for this $(p-1)$-dimensional
subproblem.

\subsection{The order of oracle inputs}\label{subsec:order}

The skeletal algorithms \ref{alg:outer} and \ref{alg:inner} do not specify
which non-final vertex (facet) should be passed to the oracle; as far as the
algorithm is concerned, any choice is good. But the actual choice can have a
significant effect on the overall performance: it can affect the size
(number of vertices and facets) of the subsequent approximating polytopes, and,
which is equally important, numerical stability. There are no theoretical
results which would recommend any particular choice. Bremner in
\cite{bremner} gave an example where {\em any choice} leads to a blow-up in
the approximating polytope.

The implementation of Algorithm \ref{alg:inner} reported in Section
\ref{subsec:numresult} offers
three possibilities for choosing which vertex is to be added to the
approximation polytope.
\begin{enumerate}
\item Ask the oracle the facets in the order they were created (first facets
from the initial approximation $\S=\S_0$, then facets from $\S_1$, etc.),
and then use the returned vertex to create the next approximation.
\item Pick a non-final facet from the current approximation randomly 
with (roughly) equal probability, and ask this facet from the oracle.
\item Keep a fixed number of candidate vertices in a ``pool''. (The pool
size can be set between 10 and 100.) Fill the pool by asking the oracle
non-final facets randomly. 
Give a score for each vertex in the pool and choose the
vertex with the best score.
\end{enumerate}
The following heuristics gave consistently the best result: choose the vertex
which discards the largest number of facets (that is, the vertex for which
the set $H^-$ is the biggest).
It
would be interesting to see some theoretical support for this heuristics.

\subsection{Parallelizing}\label{subsec:parallel}

While the simplex algorithm solving scalar LP problems is inherently serial,
the most time-consuming part of vertex enumeration -- the ridge test -- can
be done in parallel. There are LP intensive problems -- especially when the
number of variables and constraints are high and there are only a few
objectives -- where almost the total
execution time is spent by the LP solver; and there are combinatorial
intensive problems -- typically when the number of objectives is
20 or more -- where the LP solver takes negligible time. In the latter cases
the average number of ridge tests per iteration can be as high as
$10^7$---$10^9$ and it takes hours to complete the iteration.
Doing it in parallel can reduce the execution time as
explained in Section \ref{sec:vertex-enum}.

In a multithread environment each thread is assigned a set of facet pairs on
which it executes the ridge test as defined in Proposition \ref{ridge-test},
and computes the coordinates of the new facet if the pair passes the test.
Every thread uses the current facet and vertex bitmaps with total size up to
$10^8$--$10^9$ byte, while every thread requires a relative small private
memory around $10^5$ byte. Run on a single processor the algorithm scales
well with the number of assigned threads. On supercomputers with
hundreds of threads (and processors) memory latency can be an issue
\cite{parallel}.

\begin{table}\label{table:sample}
\begin{center}\begin{tabular}{|rrr|rrr|}
\hline
\multicolumn{3}{|c}{\rule{0pt}{9pt} problem}&\multicolumn{3}{|c|}{\strut solution} \\
vars & eqs & objs & vertices & facets & time \\
\hline
\rule{0pt}{9pt}
844 & 12 & 10 & 77 & 817        & 0:01 \\
888 & 12 & 10 & 1291 & 11232    & 2:44 \\
1983& 12 & 10 & 2788 & 26859 & 6:24 \\
9472 & 707 & 10 & 97 & 271 & 19:50 \\
138 & 31 & 21 & 18 & 9076 & 0:06 \\
25 & 8 & 22 & 153 & 3000 & 34:50 \\
139 & 30 & 22 & 178 & 36784 & 4:41 \\
220 & 42 & 22 & 452 & 15949 & 5:59 \\
213 & 43 & 22 & 586 & 151474 & 1:07:16 \\
174 & 48 & 27 & 290 & 116091 & 2:23:17 \\
\hline
\end{tabular}\end{center}
\caption{Some computational results}
\end{table}


\subsection{Numerical results}\label{subsec:numresult}

An implementation of Algorithm \ref{alg:minner} is available at the github
site {\tt https://github.com/lcsirmaz/inner}, together with more than 80 MOLP
problems and their solutions. The problems come
from combinatorial optimization, have highly degenerate constraint
matrices and 
large number of objectives. Table \ref{table:sample} contains a sample of
this set.
Columns {\em variables} and {\em equations} refer to the columns and rows of
the constraint matrix $A$ in (\ref{eq:A}), and {\em objectives} is the number
of objectives. The next three columns contain the number of vertices and
facets of $\Q^+$ as well as the running time on a single desktop machine
with 8 Gbyte memory and Intel i5-4590 processor running on 3.3 GHz. There
is an inherent randomness in the running time as
the constraint matrix is permuted randomly, and during the iterations
the next processed facet is chosen randomly as explained in Section
\ref{subsec:order}. The difference in running time can be very high and it
depends on the size of the intermediate approximations.

\section*{Acknowledgment}
The author would like to thank the generous support of 
the Institute of Information Theory and Automation of the CAS, Prague.
The research reported in this paper has been supported by GACR project
number 16-12010S, and partially by the Lend\"ulet program of the HAS.

The careful and thorough report of the reviewers helped to improve the
presentation, clarify vague ideas, and correcting my sloppy terminology. I am
very much indebted for their valuable work.

\end{document}